\documentclass[reqno]{amsart}
\usepackage{amsmath, amssymb, amsthm, epsfig}
\usepackage{hyperref, latexsym}
\usepackage{url}
\usepackage[mathscr]{euscript}

\usepackage{color}
\usepackage{fullpage} 
\usepackage{setspace}

\def\today{\ifcase\month\or
  January\or February\or March\or April\or May\or June\or
  July\or August\or September\or October\or November\or December\fi
  \space\number\day, \number\year}

 \newtheorem{theorem}{Theorem}
  
 \newtheorem{lemma}[theorem]{Lemma}

 \theoremstyle{definition}

 \theoremstyle{remark}

 \newcommand{\C}{\mathbb{C}}
 \newcommand{\R}{\mathbb{R}}

 \newcommand{\Z}{\mathbb{Z}}

 \newcommand{\hh}{\tfrac12}
 \newcommand{\ds}{\text{\rm d}s}

  \renewcommand{\d}{\text{\rm d}}
 \newcommand{\du}{\text{\rm d}u}

 \newcommand{\dw}{\text{\rm d}w}
 \newcommand{\dx}{\text{\rm d}x}

\newcommand{\im}{{\rm Im}\,}
\newcommand{\re}{{\rm Re}\,}

\newcommand{\dd}{\,{\rm d}}
\newcommand{\meio}{\frac{1}{2}}
\newcommand{\logfeio}{\log\log C(t,\pi)^{3/d}}

\begin{document}
\title[Bounding $S_n(t)$ on RH]{Bounding $S_n(t)$ on the Riemann hypothesis}
\author[Carneiro and Chirre]{Emanuel Carneiro and Andr\'{e}s Chirre}
\address{IMPA - Instituto Nacional de Matem\'{a}tica Pura e Aplicada - Estrada Dona Castorina, 110, Rio de Janeiro, RJ, Brazil 22460-320}
\email{carneiro@impa.br}
\email{achirre@impa.br}

\allowdisplaybreaks
\numberwithin{equation}{section}

\maketitle

\begin{abstract}
Let $S(t) = \tfrac{1}{\pi} \arg \zeta \big(\hh + it \big)$ be the argument of the Riemann zeta-function at the point $\tfrac12 + it$. For $n \geq 1$ and $t>0$ define its iterates  
\begin{equation*}
S_n(t) = \int_0^t S_{n-1}(\tau) \,\d\tau\, + \delta_n\,, 
\end{equation*}
where $\delta_n$ is a specific constant depending on $n$ and $S_0(t) := S(t)$. In 1924, J. E. Littlewood proved, under the Riemann hypothesis (RH), that $S_n(t) = O(\log t/ (\log \log t)^{n+1})$. The order of magnitude of this estimate was never improved up to this date. The best bounds for $S(t)$ and $S_1(t)$ are currently due to Carneiro, Chandee and Milinovich. In this paper we establish, under RH, an explicit form of this estimate 
\begin{equation*}
-\left( C^-_n + o(1)\right) \frac{\log t}{(\log \log t)^{n+1}} \ \leq \ S_n(t) \ \leq \ \left( C^+_n + o(1)\right) \frac{\log t}{(\log \log t)^{n+1}}\,,
\end{equation*}
for all $n\geq 2$, with the constants $C_n^{\pm}$ decaying exponentially fast as $n \to \infty$. This improves   (for all $n \geq 2$) a result of Wakasa, who had previously obtained such bounds with constants tending to a stationary value when $n \to \infty$. Our method uses special extremal functions of exponential type derived from the Gaussian subordination framework of Carneiro, Littmann and Vaaler for the cases when $n$ is odd, and an optimized interpolation argument for the cases when $n$ is even. In the final section we extend these results to a general class of $L$-functions.

\end{abstract}

\section{Introduction}

This work is inserted in the fields of analytic number theory, harmonic analysis and approximation theory. Our main goal here is to improve, under the Riemann hypothesis, the known upper and lower bounds for the moments $\{S_n(t)\}_{n\geq2}$ of the argument of the Riemann zeta-function on the critical line, extending the work of Carneiro, Chandee and Milinovich \cite{CCM} for $S(t)$ and $S_1(t)$. Our argument relies on the use of certain extremal majorants and minorants of exponential type derived from the Gaussian subordination framework of Carneiro, Littmann and Vaaler \cite{CLV}. 

\medskip

Let us start by recalling the main objects of our study and some of the previous works on the topic.

\subsection{Background} Let $\zeta(s)$ denote the Riemann zeta-function. If $t$ is not the ordinate of a zero of $\zeta(s)$ we define
$$S(t) = \tfrac{1}{\pi} \arg \zeta \big(\hh + it \big),$$
where the argument is obtained by a continuous variation along straight line segments joining the points $2$, $2+i t$ and $\hh + it$, with the convention that $\arg \zeta(2) = 0$. If $t$ is the ordinate of a zero of $\zeta(s)$ we define
$$S(t) = \tfrac{1}{2}\, \lim_{\varepsilon \to 0} \left\{ S(t + \varepsilon) + S(t - \varepsilon)\right\}.$$
The function $S(t)$ has an intrinsic oscillating character and is naturally connected to the distribution of the non-trivial zeros of $\zeta(s)$ via the relation
$$N(t) = \frac{t}{2\pi} \log \frac{t}{2\pi} - \frac{t}{2\pi} + \frac{7}{8}  + S(t) + O\left( \frac{1}{t}\right),$$
where $N(t)$ counts (with multiplicity) the number of zeros $\rho = \beta + i \gamma$ of $\zeta(s)$ such that $0 < \gamma \leq t$ (zeros with ordinate $\gamma = t$ are counted with weight $\hh$). 

\medskip

Useful information on the qualitative and quantitative behavior of $S(t)$ is encoded in its moments $S_n(t)$. Setting $S_0(t) = S(t)$ we define, for $n\geq 1$ and $t >0$,
\begin{equation}\label{Intro_eq1_int_S_n}
S_n(t) = \int_0^t S_{n-1}(\tau) \,\d\tau\, + \delta_n\,,
\end{equation}
where $\delta_n$ is a specific constant depending on $n$. These are given by (see for instance \cite[p.\,2]{F1})
$$\delta_{2k-1} =\frac{ (-1)^{k-1}}{\pi} \int_{\tfrac{1}{2}}^{\infty} \int_{\sigma_{2k-2}}^{\infty} \ldots \int_{\sigma_{2}}^{\infty} \int_{\sigma_{1}}^{\infty} \log |\zeta(\sigma_0)|\, \d\sigma_0\,\d\sigma_1\,\ldots \,\d \sigma_{2k-2} $$
for $n = 2k-1$, with $k\geq 1$, and
$$\delta_{2k} = (-1)^{k-1} \int_{\tfrac{1}{2}}^{1} \int_{\sigma_{2k-1}}^{1} \ldots \int_{\sigma_{2}}^{1} \int_{\sigma_{1}}^{1} \d\sigma_0\,\d\sigma_1\,\ldots \,\d \sigma_{2k-1}  = \frac{(-1)^{k-1}}{(2k)! \cdot 2^{2k}}$$
for $n = 2k$, with $k \geq 1$.

\medskip

A classical result of Littlewood \cite[Theorem 11]{L} states that, under the Riemann hypothesis (RH),
\begin{equation}\label{Littlewood_bound}
S_n(t) = O \left( \frac{\log t}{(\log \log t)^{n+1}}\right)
\end{equation}
for $n \geq 0$. The order of magnitude of \eqref{Littlewood_bound} has not been improved over the last ninety years, and the efforts have hence been concentrated in optimizing the values of the implicit constants. In the case $n=0$, the best bound under RH is due to Carneiro, Chandee and Milinovich \cite{CCM} (see also \cite{CCM2}), who established that
\begin{equation}\label{Intro_S_t_bound}
|S(t)| \leq \left( \frac{1}{4} + o(1) \right) \frac{\log t}{\log \log t}.
\end{equation}
This improved upon earlier works of Goldston and Gonek \cite{GG}, Fujii \cite{F2} and Ramachandra and Sankaranarayanan \cite{RS}, who had obtained \eqref{Intro_S_t_bound} with constants $C =1/2$, $C=0.67$ and $C=1.12$, respectively, replacing the constant $C= 1/4$.

\medskip

For $n=1$ the current best bound under RH is also due to Carneiro, Chandee and Milinovich \cite{CCM}, who showed that 
\begin{equation}\label{Intro_S_1_t_bound}
-\left( \frac{\pi}{24} + o(1) \right) \frac{\log t}{(\log \log t)^{2}} \ \leq \ S_1(t)\  \leq \ \left( \frac{\pi}{48} + o(1) \right) \frac{\log t}{(\log \log t)^{2}}.
\end{equation}
This improved upon earlier works of Fujii \cite{F3}, and Karatsuba and Korol\"{e}v \cite{KK}, who had obtained \eqref{Intro_S_1_t_bound} with the pairs of constants $(C^+, C^-) = (0.32,0.51)$ and $(C^+, C^-) = (40,40)$, respectively, replacing the pair $(C^+, C^-) = (\pi/48, \pi/24)$.

\medskip

For $n\geq 2$, under RH, it was recently established by Wakasa \cite{W} that   
\begin{equation}\label{Intro_S_n_t_bound}
|S_n(t)| \leq \left( W_n + o(1)\right) \frac{\log t}{(\log \log t)^{n+1}},
\end{equation}
with the constant $W_n$ given by
\begin{align*}
\begin{split}
W_n & = \frac{1}{2\pi n!} \left\{ \frac{1}{1 - \tfrac{1}{e}\left(1 + \tfrac{1}{e}\right)} \sum_{j=0}^n \frac{n!}{(n-j)!} \left( \frac{1}{e} + \frac{1}{2^{j+1} e^2}\right) \right. \\
& \ \ \ \ \ \ \ \ \ \ \ \ \ \ \ \ \ \ \ \ \ \ \ \ \ \ \ \  \ \ \ \ \ \ \   \left.+ \frac{1}{(n+1)} \cdot \frac{\frac{1}{e} \left( 1+ \frac{1}{e}\right)}{1 - \frac{1}{e} \left(1 + \frac{1}{e}\right)} + \frac{1}{n(n+1)}\cdot \frac{1}{1 - \frac{1}{e}\left(1 + \frac{1}{e}\right)} \right\}
\end{split}
\end{align*}
if $n$ is odd, and 
\begin{align*}
\begin{split}
W_n & = \frac{1}{2\pi n!} \left\{ \frac{1}{1 - \tfrac{1}{e}\left(1 + \tfrac{1}{e}\right)} \sum_{j=0}^n \frac{n!}{(n-j)!} \left( \frac{1}{e} + \frac{1}{2^{j+1} e^2}\right) \right.\\
& \ \ \ \ \ \ \ \ \ \ \ \ \ \ \ \ \ \ \ \ \ \ \ \ \ \ \ \  \ \ \ \ \ \ \  \left. + \frac{1}{(n+1)} \cdot \frac{\frac{1}{e} \left( 1+ \frac{1}{e}\right)}{1 - \frac{1}{e} \left(1 + \frac{1}{e}\right)} + \frac{\pi}{2} \cdot \frac{1}{1 - \frac{1}{e}\left(1 + \frac{1}{e}\right)} \right\}
\end{split}
\end{align*}
if $n$ is even. 

\medskip

Unconditionally, it is known that $S(t) = O(\log t)$, $S_1(t) = O(\log t)$ and $S_n(t) = O\big(t^{n-1}/\log t\big)$ for $n\geq 2$ (see, for instance, \cite{F1} for the latter). In fact, the Riemann hypothesis is equivalent to the statement that $S_n(t) = o(t^{n-2})$ as $t \to \infty$, for any $n \geq 3$ (see \cite[Theorem 4]{F1}).

\subsection{Main result} Here we extend the methods of \cite{CCM} to significantly improve the bound \eqref{Intro_S_n_t_bound}. Our main result is the following.
\begin{theorem}\label{Thm1}
Assume the Riemann hypothesis. For $n \geq 0$ and $t$ sufficiently large we have
\begin{equation}\label{Intro_Thm_1_eq}
-\left( C^-_n + o(1)\right) \frac{\log t}{(\log \log t)^{n+1}} \ \leq \ S_n(t) \ \leq \ \left( C^+_n + o(1)\right) \frac{\log t}{(\log \log t)^{n+1}}\,,
\end{equation}
where $C_n^{\pm}$ are positive constants given by:
\begin{itemize}
\item For $n=0$,
$$C_0^{\pm} = \frac{1}{4}.$$

\smallskip

\item For $n = 4k +1$, with $k \in \Z^+$,
$$C_{n}^- = \frac{\zeta(n+1)}{\pi \cdot 2^{n+1}} \ \ \ \ {\rm and} \ \ \ \ C_{n}^+ = \frac{\left( 1 - 2^{-n}\right)\zeta(n+1)}{\pi \cdot 2^{n+1}}.$$

\smallskip

\item For $n = 4k +3$, with $k \in \Z^+$,
$$C_{n}^- = \frac{\left( 1 - 2^{-n}\right)\zeta(n+1)}{\pi \cdot 2^{n+1}} \ \ \ \ {\rm and} \ \ \ \ C_{n}^+ = \frac{\zeta(n+1)}{\pi \cdot 2^{n+1}}.$$

\smallskip

\item For $n \geq 2$ even,
\begin{align}\label{Bound_Thm1_n_even}
\begin{split}
C_n^+ = C_n^- & =  \left[\frac{2 \big(C_{n+1}^{+} + C_{n+1}^{-}\big) \ C_{n-1}^{+}\ C_{n-1}^{-}}{C_{n-1}^{+} + C_{n-1}^{-}}\right]^{1/2} \\
& = \frac{\sqrt{2}}{\pi \cdot 2^{n+1}}\left[\frac{\left(1 - 2^{-n-2}\right)\,\left( 1 - 2^{-n+1}\right)\,\zeta(n) \ \zeta(n+2)}{\left(1 - 2^{-n}\right)}\right]^{1/2}.
\end{split}
\end{align}
\end{itemize}
The terms $o(1)$ in \eqref{Intro_Thm_1_eq} are $O(\log \log \log t / \log \log t)$.\footnote{We remark that the implicit constants in the $O-$notation in our estimates (as well as in \eqref{Littlewood_bound}) are allowed to depend on $n$.}  
\end{theorem}
For $n=0$ and $n=1$ this is a restatement of the result of Carneiro, Chandee and Milinovich \cite{CCM}. The novelty here are the cases $n\geq 2$. Observe that $C_{n}^{\pm} \sim \frac{1}{\pi \cdot 2^{n+1}}$ when $n$ is odd and large and $C_{n}^{\pm} \sim \frac{\sqrt{2}}{\pi \cdot 2^{n+1}}$ when $n$ is even and large. We highlight the contrast between these exponentially decaying bounds and the previously known bounds \eqref{Intro_S_n_t_bound} of Wakasa \cite{W} that verify
$$\lim_{n\to \infty} W_n = \frac{1}{2\pi \left(1 - \tfrac{1}{e}\left(1 + \tfrac{1}{e}\right)\right)} = 0.3203696...$$
Table 1 puts in perspective the new bounds of our Theorem \ref{Thm1} and the previously known bounds \eqref{Intro_S_n_t_bound} in the small cases $2 \leq n \leq 10$. The last column reports the improvement factor.
\begin{table}
	\begin{center}
		\begin{tabular}{|l|l|l|l|l|l|}
			\hline
			$n$ & $C_n^{-}$ & $C_n^{+}$ & $W_n$ & $W_n\,/\max\{C_n^{-},C_n^{+}\}$\\
			\hline \hline
			2 & 0.0593564... & 0.0593564... & 0.6002288... & 10.1122762...\\ \hline
			3 & 0.0188406... & 0.0215321... & 0.3426156... & 15.9118250...\\ \hline
			4 & 0.0141490... & 0.0141490... & 0.3509932... & 24.8069103...\\ \hline
			5 & 0.0050598... & 0.0049017... & 0.3254151... & 64.3131985...\\ \hline
			6 & 0.0035192... & 0.0035192... & 0.3235655... & 91.9420229...\\ \hline
			7 & 0.0012387... & 0.0012484... & 0.3216216... & 257.6130647...\\ \hline
			8 & 0.0008792... & 0.0008792... & 0.3210078... & 365.0786196...\\ \hline
			9 & 0.0003111... & 0.0003105... & 0.3206826... & 1030.6078264...\\ \hline
			10& 0.0002198... & 0.0002198... & 0.3205263... & 1458.2249832...\\ \hline	
\end{tabular}
\vspace{0.2cm}
		\caption{Comparison for $2 \leq n \leq 10$.}
		\end{center}
	\end{table}
	
\subsection{Strategy outline} Our approach is partly motivated (in the case of $n$ odd) by the ideas of Goldston and Gonek \cite{GG}, Chandee and Soundararajan \cite{CS}, and Carneiro, Chandee and Milinovich \cite{CCM}, on the use of the Guinand-Weil explicit formula on special functions with compactly supported Fourier transforms (drawn from \cite{V},  \cite{CV2} and \cite{CL,CLV} respectively) to bound objects related to the Riemann zeta-function.

\medskip

The first step is to identify certain particular functions of a real variable naturally connected to the moments $S_n(t)$. For each $n\geq 0$ define a normalized function $f_n:\R \to \R$ as follows:

\smallskip

\begin{itemize}
\item If $n = 2m$, for $m \in \Z^+$, we define
\begin{equation}\label{Def_f_2m}
f_{2m}(x)=(-1)^m x^{2m}\arctan\left(\frac{1}{x}\right) +\sum_{k=0}^{m-1}\frac{(-1)^{m-k+1}}{2k+1} \,x^{2m-2k-1} -\dfrac{x}{(2m+1)(1+x^2)}.
\end{equation}  
\smallskip
\item If $n = 2m+1$, for $m \in \Z^+$, we define
\begin{equation}\label{Def_f_2m+1}
f_{2m+1}(x)=\dfrac{1}{(2m+1)}\left[(-1)^{m+1}x^{2m+1}\arctan\left(\frac{1}{x}\right) + \sum_{k=0}^{m}\dfrac{(-1)^{m-k}}{2k+1}x^{2m-2k}\right].
\end{equation}
\end{itemize}
\smallskip
We show in Lemma \ref{Rep_lem} below that, under RH, $S_n(t)$ can be expressed in terms of the sum of a translate of $f_n$ over the ordinates of the non-trivial zeros of $\zeta(s)$. From the power series representation (around the origin)
$$\arctan x = \sum_{k=0}^{\infty} \frac{(-1)^k}{2k+1}\,x^{2k+1}$$
one can check that $f_{2m}(x) \ll_m |x|^{-3}$ and $f_{2m+1}(x) \ll_m |x|^{-2}$ as $|x| \to \infty$. This rather innocent piece of information is absolutely crucial in our argument.

\medskip

Since $f_n$ is of class $C^{n-1}$ but not higher (the $n$-th derivative of $f_n$ is discontinuous at $x=0$) it will be convenient to replace $f_n$ by one-sided entire approximations of exponential type in a way that minimizes the $L^1(\R)-$error. This is the so called {\it Beurling-Selberg extremal problem} in approximation theory. These special functions have been useful in several classical applications in number theory (see for instance the excellent survey \cite{V} by J. D. Vaaler and some of the references therein) and have recently been used in connection to the theory of the Riemann zeta-function in \cite{CC, CCLM, CCM, CCM2, CF, CS, G, GG}. 

\medskip

We shall see that the even functions $f_{2m+1}$, for $m \in \Z^+$, fall under the scope of the Gaussian subordination framework of \cite{CLV}. This yields the desired existence and qualitative description of the Beurling-Selberg extremal functions in these cases (Lemma \ref{lema extremal} below) and ultimately leads to the bounds of Theorem \ref{Thm1} for $n$ odd. When $n$ is even, our argument is subtler since the functions $f_{2m}$ are odd. The Gaussian subordination framework for odd functions \cite{CL} only allows us to solve the Beurling-Selberg problem for a class of functions {\it with a discontinuity at the origin}. This is the case, for example, with the function $f_0(x) = \arctan(1/x) - x/(1+x^2)$, and this was explored in \cite{CCM} to show \eqref{Intro_S_t_bound}. For $m \geq 1$, the functions $f_{2m}$ are all odd and continuous, and the solution of the Beurling-Selberg problem for these functions is quite a delicate issue and currently unknown. We are then forced to take a very different path in this case. Having obtained \eqref{Intro_Thm_1_eq} for all odd $n$'s, we proceed with an interpolation argument to obtain the estimate for the even $n$'s in between, exploring the smoothness of $S_n(t)$ via the mean value theorem and solving two optimization problems to arrive at the bound \eqref{Bound_Thm1_n_even}.

\subsection{Extension to $L-$functions} In Section \ref{Sec_L_functions} we briefly present the extension of Theorem \ref{Thm1} to a general class of $L-$functions. In particular, this includes the Dirichlet $L$-functions $L(\cdot, \chi)$ for primitive characters $\chi$.

\section{Representation lemma}

Our starting point is the following result contained in the work of Fujii \cite{F1}.

\begin{lemma}\label{Lem2}
Assume the Riemann hypothesis. For $n \geq 0$ and $t >0$ $($$t$ not coinciding with the ordinate of a zero of $\zeta(s)$ when $n=0$$)$ we have
\begin{equation}\label{Lem1_eq_1}
S_n(t) = -\frac{1}{\pi} \,\,\im{\left\{\dfrac{i^{n}}{n!}\int_{1/2}^{\infty}{\left(\sigma-\tfrac{1}{2}\right)^{n}\,\frac{\zeta'}{\zeta}(\sigma+it)}\,\d\sigma\right\}}.
\end{equation}
\end{lemma}
\begin{proof}
This is \cite[Lemmas 1 and 2]{F1}. We provide here a brief sketch of the proof. Let $R_n(t)$ be the expression on the right-hand side of \eqref{Lem1_eq_1}. The validity of the formula for $n=0$ is clear. Proceeding by induction, let us assume that the result holds for $n = 0,1,2, \ldots, m-1$. Differentiating under the integral sign and using integration by parts one can check that $R_m'(t) = R_{m-1}(t) = S_{m-1}(t)$ (for $m=1$ we may restrict ourselves to the case when $t$ does not coincide with the ordinate of a zero of $\zeta(s)$). From \eqref{Intro_eq1_int_S_n} it remains to show that $\lim_{t \to 0^+}R_m(t) = \delta_m$ for $m \geq 1$. This follows by integrating by parts $m$ times and then taking the limit as $t \to 0^+$.
\end{proof}
The next result establishes the connection between $S_n$ and the functions $f_n$ defined in \eqref{Def_f_2m} - \eqref{Def_f_2m+1}. In the proof of Theorem \ref{Thm1} we shall only use the case of $n$ odd, but we state here the representation for $n$ even as well, as a result of independent interest.

\begin{lemma}[Representation lemma]\label{Rep_lem}
For each $n\geq 0$ let $f_n:\R \to \R$ be defined as in \eqref{Def_f_2m} - \eqref{Def_f_2m+1}. Assume the Riemann hypothesis. For $t\geq2$ $($and $t$ not coinciding with an ordinate of a zero of $\zeta(s)$ in the case $n=0$$)$ we have:
\begin{itemize}
\item[(i)] If $n = 2m$, for $m \in \Z^+$, then
\begin{equation} \label{Lem2_eq_representation_even}
S_{2m}(t)= \frac{(-1)^{m}}{\pi(2m)!} \,\sum_{\gamma}f_{2m}(t-\gamma)\,+\,O(1).  
\end{equation}
\item[(ii)] If $n = 2m+1$, for $m \in \Z^+$, then
\begin{equation} \label{Lem2_eq_representation_odd}
S_{2m+1}(t)=\frac{(-1)^{m}}{2\pi(2m+2)!}\log t -\frac{(-1)^{m}}{\pi(2m)!} \sum_{\gamma}f_{2m+1}(t-\gamma) \,+ \,O(1).
\end{equation}
\end{itemize}
The sums in \eqref{Lem2_eq_representation_even} and \eqref{Lem2_eq_representation_odd} run over the ordinates of the non-trivial zeros $\rho = \tfrac12 + i \gamma$ of $\zeta(s)$. 
\end{lemma}

\begin{proof} We split the proof into two cases: $n$ odd and $n$ even.

\medskip

\noindent {\it Case 1. $n$ odd:} Write $n = 2m +1$. It follows from Lemma \ref{Lem2} and integration by parts that 
\begin{align}\label{expansion_S_m_odd}
\begin{split}
S_{2m+1}(t) &= -\frac{1}{\pi} \,\,\im {\left\{\frac{i^{2m+1}}{(2m+1)!}\int_{1/2}^{\infty}{\left(\sigma-\tfrac{1}{2}\right)^{2m+1}\,\frac{\zeta'}{\zeta}(\sigma+it)}\,\d\sigma\right\}}\\
& = \frac{(-1)^{m+1}}{\pi (2m+1)!}\, \re {\left\{\int_{1/2}^{\infty}{\left(\sigma-\tfrac{1}{2}\right)^{2m+1}\,\frac{\zeta'}{\zeta}(\sigma+it)}\,\d\sigma\right\}}\\
& = \frac{(-1)^{m}}{\pi (2m)!}\, \re {\left\{\int_{1/2}^{\infty}{\left(\sigma-\tfrac{1}{2}\right)^{2m}\,\log \zeta(\sigma+it)}\,\d\sigma\right\}}\\
& = \frac{(-1)^{m}}{\pi (2m)!}\, \left\{\int_{1/2}^{3/2}{\left(\sigma-\tfrac{1}{2}\right)^{2m}\,\log |\zeta(\sigma+it)}|\,\d\sigma\right\} + O(1).
\end{split}
\end{align}
The idea is to replace the integrand by an absolutely convergent sum over the zeros of $\zeta(s)$ and then integrate term-by-term. We consider Riemann's $\xi-$function, defined by 
$$\xi(s)=\frac{1}{2}\,s\,(s-1)\,\pi^{-s/2}\,\Gamma\left(\frac{s}{2}\right)\,\zeta(s).$$ 
The function $\xi(s)$ is entire of order $1$ and the zeros of $\xi(s)$ correspond to the non-trivial zeros of $\zeta(s)$. By Hadamard's factorization formula (cf. \cite[Chapter 12]{Dav}), we have
\begin{equation}\label{Had_fact}
\xi(s)=e^{A+Bs}\displaystyle\prod_{\rho}\bigg(1-\frac{s}{\rho}\bigg)e^{s/\rho}\,,
\end{equation}
where $\rho = \beta + i \gamma$ runs over the non-trivial zeros of $\zeta(s)$, A is a constant and $B=-\sum_{\rho}\re(1/\rho)$. Note that $\re(1/\rho)$ is positive and that $\sum_{\rho}\re(1/\rho)$ converges absolutely. 

\medskip

Assuming the Riemann hypothesis, it follows that
\begin{equation}\label{test_point_1}
\left|\dfrac{\xi(\sigma+it)}{\xi(\tfrac32+it)}\right|=\displaystyle\prod_{\gamma}\left(\dfrac{\big(\sigma-\tfrac12)^2+(t-\gamma)^2}{1+(t-\gamma)^2}\right)^{\frac12}.
\end{equation}
Hence
$$
\log|\xi(\sigma+it)|-\log\left|\xi\left(\tfrac{3}{2}+it\right)\right|= \dfrac{1}{2}\displaystyle\sum_{\gamma}\log\left(\dfrac{\big(\sigma-\tfrac12)^2+(t-\gamma)^2}{1+(t-\gamma)^2}\right).
$$
By Stirling's formula for $\Gamma(s)$ (cf. \cite[Chapter 10]{Dav}) we obtain
\begin{align} \label{Stirling_1}
\log|\zeta(\sigma+it)|=\left(\tfrac{3}{4}-\tfrac{\sigma}{2}\right)\log t - \frac{1}{2}\displaystyle\sum_{\gamma}\log\left(\frac{1+(t-\gamma)^2}{\big(\sigma-\frac12)^2+(t-\gamma)^2}\right) + O(1), 
\end{align}
uniformly for $1/2\leq \sigma \leq 3/2$ and $t\geq 2$. Inserting \eqref{Stirling_1} into \eqref{expansion_S_m_odd} yields 
\begin{align}\label{Lem_2_eq_2_formula_S_2m+1}
\begin{split}
S_{2m+1}(t)  & = \dfrac{(-1)^{m}}{\pi(2m)!}\left(\int_{1/2}^{3/2}\left(\sigma-\tfrac{1}{2}\right)^{2m}\left(\tfrac{3}{4}-\tfrac{\sigma}{2}\right) \,\d \sigma \right)\log t\\
& \ \ \ \ \ \ \ \ \ \ \ -\frac{(-1)^{m}}{2\pi(2m)!} \int_{1/2}^{3/2}{\sum_{\gamma}\left(\sigma-\tfrac{1}{2}\right)^{2m}\log\left(\dfrac{1+(t-\gamma)^2}{(\sigma-\tfrac12)^2+(t-\gamma)^2}\right)} \,\d\sigma + O(1) \\
& = \frac{(-1)^{m}}{2\pi(2m+2)!}\log t -\frac{(-1)^{m}}{2\pi(2m)!} \sum_{\gamma}\int_{1/2}^{3/2}\left(\sigma-\tfrac{1}{2}\right)^{2m}\log\left(\dfrac{1+(t-\gamma)^2}{(\sigma-\tfrac12)^2+(t-\gamma)^2}\right) \,\d\sigma + O(1) \\
              & = \dfrac{(-1)^{m}}{2\pi(2m+2)!}\log t-\dfrac{(-1)^{m}}{\pi(2m)!} \sum_{\gamma} f_{2m+1}(t-\gamma) + O(1),   
\end{split}       
\end{align}
where the function $f_{2m+1}$ is (momentarily) defined by
\begin{align}\label{definition_f_2m+1_2}
f_{2m+1}(x)=\frac{1}{2}\int_{1/2}^{3/2}{\left(\sigma-\tfrac{1}{2}\right)^{2m}\log\left(\dfrac{1+x^2}{(\sigma-1/2)^2+x^2}\right)} \,\d\sigma\,,  
\end{align}
and the interchange between the sum and integral in \eqref{Lem_2_eq_2_formula_S_2m+1} is justified by monotone convergence since all the terms involved are nonnegative. Starting from \eqref{definition_f_2m+1_2}, a change of variables and the use of formula \cite[2.731]{GR} yield
\begin{align*}
f_{2m+1}(x) & =\frac{1}{2}\int_{0}^{1}{\sigma}^{2m}\log\left(\dfrac{1+x^2}{\sigma^2+x^2}\right) \,\d\sigma \\
& =\frac{\log(1+x^2)}{2(2m+1)}-\frac{1}{2}\int_{0}^{1}{\sigma}^{2m}\log(\sigma^2+x^2) \,\d\sigma \\
& = \frac{\log(1+x^2)}{2(2m+1)}-\frac{1}{2(2m+1)}\bigg[\sigma^{2m+1}\log(\sigma^2+x^2)+(-1)^m2x^{2m+1}\arctan\Big(\dfrac{\sigma}{x}\Big) \\
&  \ \ \ \ \ \ \ \ \ \ \ \ \ \  \ \ \ \ \     -2\displaystyle\sum_{k=0}^{m}\dfrac{(-1)^{m-k}}{2k+1}x^{2m-2k}\sigma^{2k+1}\bigg]\Bigg|^{1}_{0}\\
& = \dfrac{1}{(2m+1)}\left[(-1)^{m+1}x^{2m+1}\arctan\left(\frac{1}{x}\right) + \sum_{k=0}^{m}\dfrac{(-1)^{m-k}}{2k+1}x^{2m-2k}\right].
\end{align*}
This shows that the two definitions \eqref{Def_f_2m+1} and \eqref{definition_f_2m+1_2} agree, which completes the proof in this case.
  
\medskip

\noindent {\it Case 2. $n$ even:} Write $n = 2m$. From Lemma \ref{Lem2} it follows that
\begin{align}
\begin{split}\label{Lem_2_case_2_S_2m}
S_{2m}(t) & = -\frac{1}{\pi} \,\,\im{\left\{\frac{i^{2m}}{(2m)!}\int_{1/2}^{\infty}{\left(\sigma-\tfrac{1}{2}\right)^{2m}\,\frac{\zeta'}{\zeta}(\sigma+it)}\,\d\sigma\right\}}\\
	& = \dfrac{(-1)^{m+1}}{\pi(2m)!}\,\,\im{\left\{\int_{1/2}^{3/2}{\left(\sigma-\tfrac{1}{2}\right)^{2m}\,\frac{\zeta'}{\zeta}(\sigma+it)}\,\d\sigma\right\}} + O(1).
\end{split}
\end{align}
	We again replace the integrand by an absolutely convergent sum over the non-trivial zeros of $\zeta(s)$. Let $s=\sigma+it$. If $s$ is not a zero of $\zeta(s)$, then the partial fraction decomposition for $\zeta'(s)/\zeta(s)$ (cf. \cite[Chapter 12]{Dav}) and Stirling's formula for $\Gamma'(s)/\Gamma(s)$ (cf. \cite[Chapter 10]{Dav}) imply that
	\begin{align}\label{cota} 
	\begin{split}
	\dfrac{\zeta'}{\zeta}(s) & = \displaystyle\sum_{\rho}\left(\frac{1}{s-\rho}+\frac{1}{\rho}\right) - \frac{1}{2}\frac{\Gamma'}{\Gamma}\bigg(\dfrac{s}{2}+1\bigg)+O(1)  \\
	                                  & = \displaystyle\sum_{\rho}\bigg(\frac{1}{s-\rho}+\frac{1}{\rho}\bigg) - \dfrac{1}{2}\log\left(\frac{t}{2}\right)+O(1) 
	\end{split}
	\end{align}
uniformly for $\frac{1}{2}\leq\sigma\leq\frac{3}{2}$ and $t\geq2$, where the sum runs over the non-trivial zeros $\rho$ of $\zeta(s)$. Assume that $t$ is not the ordinate of a zero of $\zeta(s)$. Then, from \eqref{Lem_2_case_2_S_2m}, \eqref{cota} and the Riemann hypothesis, it follows that
	\begin{align}\label{test_point_2}
	\begin{split}
	S_{2m}(t) & = \dfrac{(-1)^{m+1}}{\pi(2m)!}\,\int_{1/2}^{3/2}{\left(\sigma-\tfrac{1}{2}\right)^{2m}\,\im\!\left\{\frac{\zeta'}{\zeta}(\sigma+it)\right\}}\,\d\sigma + O(1)\\
	& = \dfrac{(-1)^{m+1}}{\pi(2m)!}\,\int_{1/2}^{3/2}{\left(\sigma-\tfrac{1}{2}\right)^{2m}\,\im\!\left\{\frac{\zeta'}{\zeta}(\sigma+it) - \frac{\zeta'}{\zeta}\left(\tfrac32+it\right)\right\}}\,\d\sigma + O(1)\\
	          & = \dfrac{(-1)^{m}}{\pi(2m)!}\int_{1/2}^{3/2}{\left(\sigma-\tfrac{1}{2}\right)^{2m}\displaystyle\sum_{\gamma}\Bigg\{\dfrac{(t-\gamma)}{(\sigma-\frac{1}{2})^2+(t-\gamma)^2}-\dfrac{(t-\gamma)}{1+(t-\gamma)^2}\Bigg\}}\,\d\sigma + O(1) \\
	          & = \dfrac{(-1)^{m}}{\pi(2m)!}\displaystyle\sum_{\gamma}\int_{1/2}^{3/2}{\Biggl\{\dfrac{(\sigma-\frac{1}{2})^{2m}(t-\gamma)}{(\sigma-\frac{1}{2})^2+(t-\gamma)^2}-\dfrac{(\sigma-\frac{1}{2})^{2m}(t-\gamma)}{1+(t-\gamma)^2}\Biggl\}}\,\d\sigma + O(1)\\
	          & = \dfrac{(-1)^{m}}{\pi(2m)!}\displaystyle\sum_{\gamma}\Bigg[\displaystyle\sum_{j=1}^{m}(-1)^{j+1} \dfrac{(t-\gamma)^{2j-1}}{2m-2j+1}+(-1)^m(t-\gamma)^{2m}\arctan\bigg(\dfrac{1}{t-\gamma}\bigg) \\
	          & \ \ \ \ \ \ \ \ \ \ \ \ \ \ \ \ \ \ \ \ \ \ \ \ \ \ \ -\dfrac{t-\gamma}{(2m+1)(1+(t-\gamma)^2)}\Bigg]+O(1)\\
	          & = \dfrac{(-1)^{m}}{\pi(2m)!}\displaystyle\sum_{\gamma}\Bigg[\displaystyle\sum_{k=0}^{m-1}(-1)^{m-k+1} \dfrac{(t-\gamma)^{2m-2k-1}}{2k+1}+(-1)^m(t-\gamma)^{2m}\arctan\bigg(\dfrac{1}{t-\gamma}\bigg) \\
	          & \ \ \ \ \ \ \ \ \ \ \ \ \ \ \ \ \ \ \ \ \ \ \ \ \ \ \ -\dfrac{t-\gamma}{(2m+1)(1+(t-\gamma)^2)}\Bigg]+O(1)\\
	          &=\dfrac{(-1)^{m}}{\pi(2m)!}\displaystyle\sum_{\gamma}f_{2m}(t-\gamma)+O(1)\,,      
	          \end{split}   
	          \end{align}
where the interchange between the sum and the integral is justified by dominated convergence since $f_{2m}(x) \ll_m |x|^{-3}$ as $|x| \to \infty$. Finally, if $m \geq 1$, both sides can be extended continuously when $t$ is the ordinate of a zero of $\zeta(s)$.
\end{proof}

\noindent {\sc Remark:} Observe the introduction of a test point $\tfrac32 + it$ in a couple of passages in the proof above. This seemingly innocent object is actually quite important in dealing with the convergence issues.

\section{Proof of Theorem \ref{Thm1} in the case of $n$ odd}

\subsection{Preliminaries} The sum of $f_{2m+1}(t-\gamma)$ over the non-trivial zeros in \eqref{Lem2_eq_representation_odd} is too complicated to be evaluated directly, mainly due to the fact that $f_{2m+1}$ is only of class $C^{2m}$. The key idea to prove Theorem \ref{Thm1}  in this case is to replace the function $f_{2m+1}$ in \eqref{Lem2_eq_representation_odd} by an appropriate majorant or minorant of exponential type (thus with a compactly supported Fourier transform by the Paley-Wiener theorem). We then apply the following version of the Guinand-Weil explicit formula which connects the zeros of the zeta-function and the prime powers. 

\begin{lemma}[Guinand-Weil explicit formula] \label{GW}
Let $h(s)$ be analytic in the strip $|\im{s}|\leq \tfrac12+\varepsilon$ for some $\varepsilon>0$, and assume that $|h(s)|\ll(1+|s|)^{-(1+\delta)}$ for some $\delta>0$ when $|\re{s}|\to\infty$. Let $h(w)$ be real-valued for real $w$, and let $\widehat{h}(x)=\int_{-\infty}^{\infty}h(w)e^{-2\pi ixw}\,\dw$. Then
	\begin{align*}
	\displaystyle\sum_{\rho}h\left(\frac{\rho-\frac12}{i}\right) & = h\left(\dfrac{1}{2i}\right)+h\left(-\dfrac{1}{2i}\right)-\dfrac{1}{2\pi}\widehat{h}(0)\log\pi+\dfrac{1}{2\pi}\int_{-\infty}^{\infty}h(u)\,\re{\dfrac{\Gamma'}{\Gamma}\left(\dfrac{1}{4}+\dfrac{iu}{2}\right)}\,\du \\
	 &  \ \ \ \ \ \ \ \ \ \ \ \ \ -\dfrac{1}{2\pi}\displaystyle\sum_{n\geq2}\dfrac{\Lambda(n)}{\sqrt{n}}\left(\widehat{h}\left(\dfrac{\log n}{2\pi}\right)+\widehat{h}\left(\dfrac{-\log n}{2\pi}\right)\right)\,, 
	\end{align*}
where $\rho = \beta + i \gamma$ are the non-trivial zeros of $\zeta(s)$, $\Gamma'/\Gamma$ is the logarithmic derivative of the Gamma function, and $\Lambda(n)$ is the Von-Mangoldt function defined to be $\log p$ if $n=p^m$ with $p$ a prime number and $m\geq 1$ an integer, and zero otherwise.
\end{lemma}
\begin{proof}
	The proof of this lemma follows from \cite[Theorem 5.12]{IK}. 
\end{proof}

The existence and qualitative description of the appropriate majorants and minorants of exponential type for $f_{2m+1}$ will come from the general machinery developed by Carneiro, Littmann and Vaaler \cite{CLV} to solve the Beurling-Selberg extremal problem for a class of even functions subordinated to the Gaussian. We collect the relevant properties for our purposes in the next lemma, that shall be proved in Section \ref{Extremal_functions_section}. This lemma is the generalization of \cite[Lemma 4]{CCM} that considers the case $m=0$.

\begin{lemma}[Extremal functions]\label{lema extremal} Let $m \geq0$ be an integer and let $\Delta\geq1$ be a real parameter. Let $f_{2m+1}$ be the real valued function defined in \eqref{Def_f_2m+1}, i.e. 
	\[
	f_{2m+1}(x)=\dfrac{1}{(2m+1)}\left[(-1)^{m+1}x^{2m+1}\arctan\left(\frac{1}{x}\right) + \sum_{k=0}^{m}\dfrac{(-1)^{m-k}}{2k+1}x^{2m-2k}\right].
	\]
Then there are unique real entire functions $g_{2m+1,\Delta}^{-}:\mathbb{C}\to\mathbb{C}$ and $g_{2m+1,\Delta}^{+}:\mathbb{C}\to\mathbb{C}$ satisfying the following properties:
\begin{itemize}
\item[(i)] For $x\in\mathbb{R}$ we have
	\begin{align}\label{EF_Lem_eq1}
	-\dfrac{K_{2m+1}}{1+x^2}\leq g_{2m+1,\Delta}^{-}(x) \leq f_{2m+1}(x) \leq g_{2m+1,\Delta}^{+}(x) \leq \dfrac{K_{2m+1}}{1+x^2}\,, 
	\end{align}
for some positive constant $K_{2m+1}$ independent of $\Delta$. Moreover, for any complex number $z=x+iy$ we have
\begin{align}\label{EF_Lem_eq2}
		\big|g_{2m+1,\Delta}^{\pm}(z)\big|\ll_m \dfrac{\Delta^{2}}{(1+\Delta|z|)}e^{2\pi\Delta|y|}.   
\end{align}

\smallskip

\item[(ii)] The Fourier transforms of $g_{2m+1,\Delta}^{\pm}$, namely
	\[
	\widehat{g}_{2m+1,\Delta}^{\pm}(\xi)=\int_{-\infty}^{\infty}g_{2m+1,\Delta}^{\pm}(x)\,e^{-2\pi ix\xi}\,\dx,
	\]
	are continuous functions supported on the interval $[-\Delta,\Delta]$ and satisfy
	\begin{equation}\label{FT_unif_bounded}
	\widehat{g}_{2m+1,\Delta}^{\pm}(\xi) \ll_m 1
	\end{equation}
	for all $\xi\in[-\Delta,\Delta]$, where the implied constant is independent of $\Delta$.
	
	\smallskip
	
\item[(iii)] The $L^1-$distances of $g_{2m+1,\Delta}^{\pm}$ to $f_{2m+1}$ are explicitly given by
	\begin{equation}\label{EF_Lem_eq4}
	\int_{-\infty}^{\infty}\big\{f_{2m+1}(x)-g^{-}_{2m+1,\Delta}(x)\big\}\,\dx=\dfrac{1}{\Delta}\int_{1/2}^{3/2}\left(\sigma-\tfrac{1}{2}\right)^{2m}\,\log\left(\dfrac{1+e^{-2\pi(\sigma-1/2)\Delta}}{1+e^{-2\pi\Delta}}\right)\d\sigma
	\end{equation}
	and
	\begin{equation}\label{EF_Lem_eq5}
	\int_{-\infty}^{\infty}\big\{g^{+}_{2m+1,\Delta}(x)-f_{2m+1}(x)\big\}\,\dx=
	-\dfrac{1}{\Delta}\int_{1/2}^{3/2}\left(\sigma-\tfrac{1}{2}\right)^{2m}\,\log\left(\dfrac{1-e^{-2\pi(\sigma-1/2)\Delta}}{1-e^{-2\pi\Delta}}\right)\d\sigma.
	\end{equation}
\end{itemize}
\end{lemma}

\subsection{Proof of Theorem \ref{Thm1} for $n$ odd} Let $n = 2m+1$. To simplify notation we disregard one of the subscripts and write $g^{\pm}_{\Delta}(z):=g^{\pm}_{2m+1,\Delta}(z)$. For a fixed $t >0$, we consider the functions $h^{\pm}_{\Delta}(z):=g^{\pm}_{\Delta}(t-z)$. Then $\widehat{h}^{\pm}_{\Delta}(\xi)=\widehat{g}^{\pm}_{\Delta}(-\xi)e^{-2\pi i\xi t}$ and the condition $|h^{\pm}_{\Delta}(s)|\ll(1+|s|)^{-2}$ when $|\re{s}|\to\infty$ in the strip $|\im{s}|\leq 1$ follows from \eqref{EF_Lem_eq1}, \eqref{EF_Lem_eq2} and an application of the Phragm\'{e}n-Lindel\"{o}f principle. We can then apply the Guinand-Weil explicit formula (Lemma \ref{GW}) to get
\begin{align}\label{GW_applied_to_h}
\begin{split}
\displaystyle\sum_{\gamma}g^{\pm}_{\Delta}(t-\gamma) & = 
\Big\{g^{\pm}_{\Delta}\left(t-\tfrac{1}{2i}\right)+g^{\pm}_{\Delta}\left(t+\tfrac{1}{2i}\right)\Big\}-\dfrac{1}{2\pi}\widehat{g}^{\pm}_{\Delta}(0)\log\pi \\ 
     & + \dfrac{1}{2\pi}\int_{-\infty}^{\infty}g^{\pm}_{\Delta}(t-x)\,\re\,\dfrac{\Gamma'}{\Gamma}\bigg(\dfrac{1}{4}+\dfrac{ix}{2}\bigg) \,\dx \\ 
     & - \dfrac{1}{2\pi}\displaystyle\sum_{n\geq 2}\dfrac{\Lambda(n)}{\sqrt{n}}\,\left\{ \widehat{g}^{\pm}_{\Delta}\left(-\dfrac{\log n}{2\pi}\right)\,e^{-it\log n}+\widehat{g}^{\pm}_{\Delta}\left(\dfrac{\log n}{2\pi}\right)\,e^{it\log n}\right\}. 
\end{split}
\end{align}
\subsubsection{Asymptotic analysis} We now analyze each term on the right-hand side of \eqref{GW_applied_to_h} separately.

\smallskip

\noindent 1. {\it First term}: From \eqref{EF_Lem_eq2} we get
\begin{align}\label{AsA_eq1}
\Big|g^{\pm}_{\Delta}\left(t-\tfrac{1}{2i}\right)+g^{\pm}_{\Delta}\left(t+\tfrac{1}{2i}\right)\Big| \ll_m \,\Delta^2\dfrac{e^{\pi\Delta}}{1+\Delta t}. 
\end{align}

\smallskip

\noindent 2. {\it Second term}: From \eqref{FT_unif_bounded} we get
\begin{align}\label{AsA_eq2}
\left|\frac{1}{2\pi}\widehat{g}^{\pm}_{\Delta}(0)\log\pi\right|\ll_m 1.  
\end{align} 

\smallskip

\noindent 3. {\it Fourth term}: Recall that the Fourier transforms $\widehat{g}^{\pm}_{\Delta}$ are supported on the interval $[-\Delta,\Delta]$. Using \eqref{FT_unif_bounded}, summation by parts and the Prime Number Theorem we obtain
\begin{align}\label{AsA_eq3}
\Bigg|\dfrac{1}{2\pi}\displaystyle\sum_{n\geq 2}\dfrac{\Lambda(n)}{\sqrt{n}}\,\left\{ \widehat{g}^{\pm}_{\Delta}\left(-\dfrac{\log n}{2\pi}\right)\,e^{-it\log n}+\widehat{g}^{\pm}_{\Delta}\left(\dfrac{\log n}{2\pi}\right)\,e^{it\log n}\right\}\Bigg|\ll_m \displaystyle\sum_{n\leq e^{2\pi\Delta}}\dfrac{\Lambda(n)}{\sqrt{n}}\ll_m \,e^{\pi\Delta}.
\end{align}

\smallskip

\noindent 4. {\it Third term}: This is the term that requires most of our attention. Using \eqref{definition_f_2m+1_2} and \cite[2.733 - Formula 1]{GR} we start by observing that 
\begin{align}\label{Ev_int_f_2m+1}
\begin{split}
\int_{-\infty}^{\infty}f_{2m+1}(x)\,\dx & = \dfrac{1}{2}\int_{-\infty}^{\infty}\int_{0}^{1}\sigma^{2m}\log\Bigg(\dfrac{1+x^2}{\sigma^2+x^2}\Bigg) \,\d\sigma \,\dx  \\
&  =
\dfrac{1}{2}\int_{0}^{1}{\sigma^{2m}\int_{-\infty}^{\infty}\log\Bigg(\dfrac{1+x^2}{\sigma^2+x^2}\Bigg)} \, \dx\,\d\sigma \\
& =  \dfrac{1}{2}\int_{0}^{1}\sigma^{2m}\bigg[x\log\bigg(\dfrac{1+ x^2}{\sigma^2 + x^2}\bigg)+2\arctan(x)-2\sigma\arctan\Big(\dfrac{x}{\sigma}\Big)\bigg]\Bigg|^{\infty}_{-\infty} \, \d\sigma  \\
& = \pi \int_{0}^{1}\sigma^{2m}(1-\sigma)\,\d\sigma \\
& = \dfrac{\pi}{(2m+1)(2m+2)}.
\end{split} 
\end{align}
Let us assume without loss of generality that $t \geq 10$. Using Stirling's formula for $\Gamma'/\Gamma$ (cf. \cite[Chapter 10]{Dav}), together with \eqref{EF_Lem_eq1}, \eqref{EF_Lem_eq4}, \eqref{EF_Lem_eq5} and \eqref{Ev_int_f_2m+1}, we get
\begin{align}\label{Final_eq0}
\begin{split}
& \dfrac{1}{2\pi}\int_{-\infty}^{\infty}g^{\pm}_{\Delta}(t-x)\,\re\,\dfrac{\Gamma'}{\Gamma}\bigg(\dfrac{1}{4}+\dfrac{ix}{2}\bigg) \,\dx \\
& = \dfrac{1}{2\pi}\int_{-\infty}^{\infty}g^{\pm}_{\Delta}(x)\big(\log t + O(\log(2+|x|))) \,\dx \\
& = \dfrac{1}{2\pi}\int_{-\infty}^{\infty}\Big\{f_{2m+1}(x)-\big(f_{2m+1}(x)-g^{\pm}_{\Delta}(x)\big)\Big\}\big(\log t + O(\log(2+|x|))) \,\dx  \\
& = \dfrac{\log t}{2(2m+1)(2m+2)}-
\dfrac{\log t}{2\pi\Delta}\int_{1/2}^{3/2}\left(\sigma-\tfrac{1}{2}\right)^{2m}\log\Bigg(\dfrac{1\mp e^{-2\pi(\sigma-1/2)\Delta}}{1\mp e^{-2\pi\Delta}}\Bigg)\d\sigma + O(1)  \\
& = \dfrac{\log t}{2(2m+1)(2m+2)}-
\dfrac{\log t}{2\pi\Delta}\int_{1/2}^{\infty}\left(\sigma-\tfrac{1}{2}\right)^{2m}\log\left(1\mp e^{-2\pi(\sigma-1/2)\Delta}\right)\d\sigma +  O\left(e^{-\pi\Delta}\log t\right) + O(1). 
\end{split}
\end{align}
We evaluate this last integral expanding $\log(1\mp x)$ into a power series:
\begin{align*}
\int_{1/2}^{\infty} \left(\sigma-\tfrac{1}{2}\right)^{2m}  & \log\left(1  \mp e^{-2\pi(\sigma-1/2)\Delta}\right)\d\sigma  = \int_{0}^{\infty} \sigma^{2m}\log\Big(1\mp e^{-2\pi\sigma\Delta}\Big)\d\sigma \\
& = \int_{0}^{\infty} \sigma^{2m} \sum_{k\geq 0} \left\{ \mp \frac{e^{-2\pi\sigma\Delta(2k+1)}}{2k+1} -\frac{e^{-2\pi\sigma\Delta(2k+2)}}{2k+2}\right\} \d\sigma\\
& = \sum_{k\geq 0}\int_{0}^{\infty} \sigma^{2m}\left\{ \mp \frac{e^{-2\pi\sigma\Delta(2k+1)}}{2k+1} -\frac{e^{-2\pi\sigma\Delta(2k+2)}}{2k+2}\right\} \d\sigma\\
& =\frac{(2m)!}{(2\pi \Delta)^{2m+1}} \sum_{k\geq 0} \left\{\mp \frac{1}{(2k+1)^{2m+2}} - \frac{1}{(2k+2)^{2m+2}}\right\}.
\end{align*}
The interchange between integral and sum above is guaranteed by the monotone convergence theorem since all terms involved have the same sign. We have thus arrived at the following two expressions:
\begin{align} \label{ASA_eq4}
\begin{split}
\dfrac{1}{2\pi}\int_{-\infty}^{\infty} & g^{+}_{\Delta}(t-x)\, \re\,\dfrac{\Gamma'}{\Gamma}\left(\dfrac{1}{4}+\dfrac{ix}{2}\right) \,\dx \\
&  = \dfrac{\log t}{2(2m+1)(2m+2)} + 
\dfrac{(2m)! \,\zeta(2m+2)}{(2\pi\Delta)^{2m+2}} \log t +  O\left(e^{-\pi\Delta}\log t\right) + O(1) 
\end{split}
\end{align}
and 
\begin{align}\label{ASA_eq5}
\begin{split}
\dfrac{1}{2\pi}\int_{-\infty}^{\infty}& g^{-}_{\Delta}(t-x)\, \re\,\dfrac{\Gamma'}{\Gamma}\bigg(\dfrac{1}{4}+\dfrac{ix}{2}\bigg) \,\dx\\
&   =  \dfrac{\log t}{2(2m+1)(2m+2)} - 
\dfrac{(2m)! \,\left( 1 - 2^{-2m-1}\right)\zeta(2m+2)}{(2\pi\Delta)^{2m+2}} \log t+  O\left(e^{-\pi\Delta}\log t\right) + O(1).
\end{split}
\end{align}

\smallskip

\subsubsection{Conclusion of the proof} Recall that $n=2m+1$. We now consider two cases:

\smallskip

\noindent \underline{{\it Case 1}: $m$ even.}

\smallskip

In this case, by \eqref{Lem2_eq_representation_odd} we have
\begin{equation*}
S_{2m+1}(t)=\frac{1}{2\pi(2m+2)!}\log t -\frac{1}{\pi(2m)!} \sum_{\gamma}f_{2m+1}(t-\gamma) \,+ \,O(1).
\end{equation*}
Using \eqref{EF_Lem_eq1} we arrive at
\begin{align*}
\frac{1}{2\pi(2m+2)!}\log t -\frac{1}{\pi(2m)!} & \sum_{\gamma}g_{2m+1, 
\Delta}^+(t-\gamma)  \,+ \,O(1) \\ 
& \leq S_{2m+1}(t) \\
&  \leq \frac{1}{2\pi(2m+2)!}\log t -\frac{1}{\pi(2m)!} \sum_{\gamma}g_{2m+1, 
\Delta}^-(t-\gamma) \,+ \,O(1).
\end{align*}
From \eqref{GW_applied_to_h}, \eqref{AsA_eq1}, \eqref{AsA_eq2}, \eqref{AsA_eq3}, \eqref{ASA_eq4} and \eqref{ASA_eq5} we find
\begin{align}\label{Concl_pf_thm_1_n_odd}
\begin{split}
 -\frac{\zeta(2m+2)}{\pi (2 \pi \Delta)^{2m+2}} \log t & + O \left(\tfrac{\Delta^2\,e^{\pi\Delta}}{1+\Delta t}\right) + O\left(e^{-\pi\Delta}\log t\right) + O\big(e^{\pi \Delta} + 1\big)\\
& \leq S_{2m+1}(t)\\
 & \leq \dfrac{\left( 1 - 2^{-2m-1}\right)\zeta(2m+2)}{\pi (2\pi\Delta)^{2m+2}} \log t +  O \left(\tfrac{\Delta^2\,e^{\pi\Delta}}{1+\Delta t}\right) + O\left(e^{-\pi\Delta}\log t\right) + O\big(e^{\pi \Delta} + 1\big).
 \end{split}
\end{align}
Choosing 
$$\pi \Delta = \log \log t - (2m+3) \log \log \log t$$ 
in \eqref{Concl_pf_thm_1_n_odd} we obtain
\begin{align*}
-\left(\frac{\zeta(2m +2)}{\pi \cdot 2^{2m+2}} + o(1) \right) \frac{\log t}{(\log \log t)^{2m+2}}  \ \leq \ S_{2m+1}(t)\ \leq \ \left(\frac{\left( 1 - 2^{-2m-1}\right)\zeta(2m+2)}{\pi \cdot 2^{2m+2}} +o(1)\right) \frac{\log t}{(\log \log t)^{2m+2}},\end{align*} 
where the terms $o(1)$ above are $O(\log \log \log t / \log \log t)$.

\smallskip

\noindent \underline{{\it Case 2}: $m$ odd.}

\smallskip

 Using \eqref{Lem2_eq_representation_odd} we get
\[
S_{2m+1}(t)=\dfrac{-1}{2\pi(2m+2)!}\log t+\dfrac{1}{\pi(2m)!} \displaystyle\sum_{\gamma}f_{2m+1}(t-\gamma) + O(1),
\]
and we only need to interchange the roles of $g_{\Delta}^{+}$ and $g_{\Delta}^{-}$ in comparison to the previous case. Similar calculations show that  
\begin{align*}
-\left(C_{2m+1}^{-} + o(1) \right)\dfrac{\log t}{(\log\log t)^{2m+2}} \ \leq \ S_{2m+1}(t)\ \leq \ \left(C_{2m+1}^{+} + o(1) \right)\dfrac{\log t}{(\log\log t)^{2m+2}},
\end{align*}
where the terms $o(1)$ above are $O(\log \log \log t / \log \log t)$ and 
$$
C_{2m+1}^{-}=\dfrac{\left( 1 - 2^{-2m-1}\right)\zeta(2m+2)}{\pi \cdot 2^{2m+2}}\ \ \ {\rm and} \ \ \ C_{2m+1}^{+}=\dfrac{\zeta(2m+2)}{\pi \cdot 2^{2m+2}}.
$$

This completes the proof of Theorem \ref{Thm1} for $n$ odd.

\section{An interpolation argument: proof of Theorem \ref{Thm1} in the case of $n$ even} \label{Int_sec}
In order to further simplify the notation let us write
\[
\ell_{n}(t):=\dfrac{\log t}{(\log\log t)^{n}} \hspace{0.5cm} \mbox{and} \hspace{0.5cm} r_{n}(t):=\dfrac{\log t\log\log\log t}{(\log\log t)^{n}}.
\]
Let $n \geq 2$ be an even integer (the case $n=0$ was established in \cite{CCM}). We have already shown that 
\begin{align}
-C_{n-1}^{-}\ell_{n}(t) + O(r_{n+1}(t)) \leq  S_{n-1}(t) \leq C_{n-1}^{+}\ell_{n}(t) + O(r_{n+1}(t)) \label{interpol_1} 
\end{align}
and
\begin{align}
-C_{n+1}^{-}\ell_{n+2}(t) + O(r_{n+3}(t)) \leq  S_{n+1}(t) \leq C_{n+1}^{+}\ell_{n+2}(t) + O(r_{n+3}(t)) \label{interpol_2}.
\end{align}
Our goal now is to obtain a similar estimate for $S_n(t)$ that interpolates between \eqref{interpol_1} and \eqref{interpol_2}. We view this as a pure analysis problem and our argument below explores the fact that the function $S_n(t)$, for $n\geq 2$, is continuously differentiable.

\medskip

By the mean value theorem and \eqref{interpol_1} we obtain, for $-\sqrt{t} \leq h \leq \sqrt{t}$, 
\begin{align}\label{prep_integral}
\begin{split}
S_{n}(t) - S_{n}(t-h) & = h \, S_{n-1}(t_h^*) \\
& \leq \left(\chi_{h>0}\,|h|\,C_{n-1}^{+}+\chi_{h<0}\,|h|\,C_{n-1}^{-}\right)\ell_{n}(t^{*}_h)+|h|\,O(r_{n+1}(t^{*}_h))\\
& \leq \left(\chi_{h>0}\,|h|\,C_{n-1}^{+}+\chi_{h<0}\,|h|\,C_{n-1}^{-}\right)\ell_{n}(t)+|h| \,O(r_{n+1}(t))\,,
\end{split}
\end{align}
where $t^*_h$ is a suitable point in the segment connecting $t-h$ and $t$, and $\chi_{h>0}$ and $\chi_{h<0}$ are the indicator functions of the sets $\{h \in \R; \,h>0\}$ and $\{h \in \R; \,h<0\}$, respectively.

\medskip

Let $a$ and $b$ be positive real numbers that shall be properly chosen later (in particular, we will be able to choose them in a way that $a + b =1$, for instance). Let $\nu$ be a real parameter such that $0 < \nu \leq \sqrt{t}$. We integrate \eqref{prep_integral} with respect to the variable $h$ to get
\begin{align*}
S_{n}(t) & \leq \frac{1}{(a+b)\nu} \int_{-a\nu}^{b\nu} S_{n}(t-h)\,\d h \  + \frac{1}{(a+b)\nu} \left[ \int_{-a\nu}^{b\nu}\left(\chi_{h>0}\,|h|\,C_{n-1}^{+}+\chi_{h<0}\,|h|\,C_{n-1}^{-}\right)\,\d h\right] \ell_{n}(t)\\
&  \ \ \ \ \ \ \ \ \ \ \ \ \ \ \ + \frac{1}{(a+b)\nu} \left[\int_{-a\nu}^{b\nu} |h|\,\d h \right] O(r_{n+1}(t))\\
& = \frac{1}{(a+b)\nu} \Big[ S_{n+1}(t + a\nu) - S_{n+1}(t - b\nu)\Big] + \left[\dfrac{b^2C_{n-1}^{+} + a^{2}C_{n-1}^{-}}{2(a+b)}\right]\nu \,\ell_{n}(t) + O(\nu \,r_{n+1}(t)).
\end{align*}
We now use \eqref{interpol_2} to get
\begin{align}\label{To_be_optimized}
\begin{split}
S_{n}(t) & \leq \frac{1}{(a+b)\nu} \Big[ C_{n+1}^{+}\ell_{n+2}(t+a\nu) + C_{n+1}^{-}\ell_{n+2}(t-b\nu) + O(r_{n+3}(t+a\nu)) + O(r_{n+3}(t-b\nu))\Big] \\
&  \ \ \ \ \ \ \ \ \ + \left[\dfrac{b^2C_{n-1}^{+} + a^{2}C_{n-1}^{-}}{2(a+b)}\right]\nu \,\ell_{n}(t) + O(\nu \,r_{n+1}(t))\\
& = \left[\frac{C_{n+1}^{+} + C_{n+1}^{-}}{(a+b)}\right] \frac{1}{\nu}\,\,\ell_{n+2}(t)+ \left[\dfrac{b^2C_{n-1}^{+} + a^{2}C_{n-1}^{-}}{2(a+b)}\right]\nu \,\ell_{n}(t) + O\left(\frac{r_{n+3}(t)}{\nu}\right) + O(\nu \,r_{n+1}(t)).
\end{split}
\end{align}

\medskip

Choosing $\nu = \frac{\alpha}{\log \log t}$ in \eqref{To_be_optimized}, where $\alpha>0$ is a constant to be determined, we find
\begin{align*}
S_{n}(t) \leq \left\{ \left[\frac{C_{n+1}^{+} + C_{n+1}^{-}}{(a+b)}\right] \frac{1}{\alpha} + \left[\dfrac{b^2C_{n-1}^{+} + a^{2}C_{n-1}^{-}}{2(a+b)}\right]\,\alpha\right\} \,\ell_{n+1}(t) + O(r_{n+2}(t)).
\end{align*}
We now choose $\alpha>0$ to minimize the expression in brackets, which corresponds to the choice
$$\alpha = \left[\frac{C_{n+1}^{+} + C_{n+1}^{-}}{(a+b)}\right]^{1/2} \left[\dfrac{b^2C_{n-1}^{+} + a^{2}C_{n-1}^{-}}{2(a+b)}\right]^{-1/2}.$$
This leads to the bound
\begin{align}\label{opt_prob_2}
S_n(t)& \leq 2  \left[\frac{\big(C_{n+1}^{+} + C_{n+1}^{-}\big)  \big( b^2C_{n-1}^{+} + a^{2}C_{n-1}^{-}\big)}{2(a+b)^2}\right]^{1/2} \ell_{n+1}(t) + O(r_{n+2}(t)).
\end{align}

\medskip

We now seek to minimize the right-hand side of \eqref{opt_prob_2} in the variables $a$ and $b$. It is easy to see that it only depends on the ratio $a/b$ (and hence we can normalize to have $a + b = 1$). If we consider $a = bx$ we must minimize the function
$$H(x) = 2  \left[\frac{\big(C_{n+1}^{+} + C_{n+1}^{-}\big)  \big( C_{n-1}^{+} + x^{2}C_{n-1}^{-}\big)}{2(x+1)^2}\right]^{1/2}.$$
Note that $C_{n-1}^{\pm} >0$ and $C_{n+1}^{\pm} >0$. Such a minimum is obtained when $x = C_{n-1}^{+} / C_{n-1}^{-}$, leading to the bound
\begin{align*}
S_n(t)& \leq \left[\frac{2 \big(C_{n+1}^{+} + C_{n+1}^{-}\big) \ C_{n-1}^{+}\ C_{n-1}^{-}}{C_{n-1}^{+} + C_{n-1}^{-}}\right]^{1/2} \ell_{n+1}(t) + O(r_{n+2}(t)).
\end{align*}
The argument for the lower bound of $S_n(t)$ is entirely symmetric. This completes the proof of Theorem \ref{Thm1}.

\section{Extremal functions via Gaussian subordination}\label{Extremal_functions_section}

In this section we give a complete proof of Lemma \ref{lema extremal}.

\subsection{Preliminaries} The problem of finding one-sided approximations of real-valued functions by entire functions of prescribed exponential type, seeking to minimize the $L^1(\R)-$error, is a classical problem in approximation theory. This problem has its origins in the works of A. Beurling and A. Selberg, who constructed majorants and minorants of exponential type for the signum function and characteristic functions of intervals, respectively. The survey \cite{V} by J. D. Vaaler is the classical reference on the subject, describing some of the historical milestones of the problem and presenting a number of interesting applications of such special functions to analysis and number theory. Over the last years there has been considerable progress on the constructive side of such special functions (see for instance \cite{CL, CLV, CV2}, and the references therein, for the one-dimensional theory and \cite{CL2, CL3, HV} for multidimensional analogues), unveiling new applications to number theory, in particular to the theory of the Riemann zeta-function \cite{CC, CCLM, CCM, CCM2, CF, CS, G, GG}, as already mentioned in the introduction. 

\medskip

The appropriate machinery for our purposes is the Gaussian subordination framework of \cite{CLV}, a method that allows one to solve the Beurling-Selberg extremal problem for a wide class of even functions. In particular, functions $g:\mathbb{R}\to\mathbb{R}$ of the form
\[
g(x)=\int_{0}^{\infty}e^{-\pi\lambda x^2}\,\d\nu(\lambda),
\]
where $\nu$ is a finite nonnegative Borel measure on $(0,\infty)$, fall under the scope of \cite{CLV}. It turns out that our functions $f_{2m+1}$ defined in \eqref{Def_f_2m+1} are included in this class. 

\medskip

In fact, for $\Delta\geq 1$, we consider the nonnegative Borel measure $\nu_{\Delta} = \nu_{2m+1, \Delta} $ on $(0,\infty)$ given by
$$
\d\nu_{\Delta}(\lambda):=\int_{1/2}^{3/2}\left(\sigma-\tfrac{1}{2}\right)^{2m}\Bigg(\dfrac{e^{-\pi\lambda(\sigma-1/2)^{2} \Delta^2}-e^{-\pi\lambda\Delta^2}}{2\lambda}\Bigg)\,\d\sigma\, \d\lambda\,,
$$
and let $F_{\Delta}=F_{2m+1,\Delta}$ be the function 
\[
F_\Delta(x):=\int_{0}^{\infty}e^{-\pi\lambda x^2}\,\d\nu_{\Delta}(\lambda).
\]
Recall that 
\[
\dfrac{1}{2}\log\Bigg(\dfrac{x^2+\Delta^2}{x^2+(\sigma-1/2)^2\Delta^2}\Bigg)=\int_{0}^{\infty}e^{-\pi\lambda x^2}\Bigg(\dfrac{e^{-\pi\lambda(\sigma-1/2)^{2} \Delta^2}-e^{-\pi\lambda\Delta^2}}{2\lambda}\Bigg)\,\d\lambda.
\]
Multiplying both sides by $(\sigma-1/2)^{2m}$ and integrating from $\sigma=1/2$ to $\sigma=3/2$ yields
\begin{align*}
\dfrac{1}{2} \int_{1/2}^{3/2} &\left(\sigma-\tfrac{1}{2}\right)^{2m}\log\Bigg(\dfrac{x^2+\Delta^2}{x^2+(\sigma-1/2)^2\Delta^2}\Bigg)\,\d\sigma \\
& =  \int_{1/2}^{3/2}\int_{0}^{\infty}\left(\sigma-\tfrac{1}{2}\right)^{2m}e^{-\pi\lambda x^2}\Bigg(\dfrac{e^{-\pi\lambda(\sigma-1/2)^{2} \Delta^2}-e^{-\pi\lambda\Delta^2}}{2\lambda}\Bigg)\,\d\lambda \,\d\sigma \\
& =\int_{0}^{\infty}e^{-\pi\lambda x^2}\int_{1/2}^{3/2}\left(\sigma-\tfrac{1}{2}\right)^{2m}\Bigg(\dfrac{e^{-\pi\lambda(\sigma-1/2)^{2} \Delta^2}-e^{-\pi\lambda\Delta^2}}{2\lambda}\Bigg)\, \d\sigma \,\d\lambda\\
& =F_{\Delta}(x), 
\end{align*}
where the interchange of the integrals is justified since the terms involved are all nonnegative. It follows from \eqref{definition_f_2m+1_2} that 
\begin{align}\label{Eq_31_F}
f_{2m+1}(x)=F_{\Delta}(\Delta x).
\end{align}
In particular, this shows that the measure $\nu_{\Delta}$ is finite on $(0,\infty)$ since
\begin{align*}
\int_{0}^{\infty}\d\nu_{\Delta}(\lambda)=F_{\Delta}(0)=f_{2m+1}(0)=\dfrac{1}{(2m+1)^2}. 
\end{align*}

\medskip

By \cite[Corollary 17]{CLV}, there is a unique extremal minorant $G^{-}_{\Delta}(z)=G^{-}_{2m+1,\Delta}(z)$ and a unique extremal majorant $G^{+}_{\Delta}(z)=G^{+}_{2m+1,\Delta}(z)$ of exponential type\footnote{Recall that an entire function $G:\C \to \C$ is said to have exponential type $\tau$ if $\limsup_{|z| \to \infty} \frac{\log |G(z)|}{|z|} \leq \tau$.} \ $2\pi$ for $F_{\Delta}(x)$, and these functions are given by
\begin{align}\label{Def_G-}
G^{-}_{\Delta}(z)=\bigg(\dfrac{\cos\pi z}{\pi}\bigg)^2 \left\{\displaystyle\sum_{n=-\infty}^{\infty}\dfrac{F_{\Delta}\big(n-\frac{1}{2}\big)}{\big(z-n+\frac{1}{2}\big)^2}+\dfrac{F^{'}_{\Delta}\big(n-\frac{1}{2}\big)}{\big(z-n+\frac{1}{2}\big)}\right\}
\end{align}
and
\begin{align}\label{Def_G+}
G^{+}_{\Delta}(z)=\bigg(\dfrac{\sin\pi z}{\pi}\bigg)^2\left\{\displaystyle\sum_{n=-\infty}^{\infty}\dfrac{F_{\Delta}(n)}{(z-n)^2}+\displaystyle\sum_{n\neq 0}\dfrac{F^{'}_{\Delta}(n)}{(z-n)}\right\}.
\end{align}
Hence, the functions  $g^{-}_{\Delta}(z)=g^{-}_{2m+1,\Delta}(z)$ and $g^{+}_{\Delta}(z)=g^{+}_{2m+1,\Delta}(z)$ defined by
\begin{align}\label{Eq_50}
g^{-}_{\Delta}(z):=G^{-}_{\Delta}(\Delta z) \hspace{0.3cm} \mbox{and} \hspace{0.3cm} g^{+}_{\Delta}(z):=G^{+}_{\Delta}(\Delta z) 
\end{align} 
are the unique extremal functions of exponential type $2\pi\Delta$ for $f_{2m+1}$. We claim that these functions verify the conditions of Lemma \ref{lema extremal}.

\subsection{Proof of Lemma \ref{lema extremal}}

\subsubsection{Part (i)} We start by observing that 
\begin{align*}
\big|f_{2m+1}(x)\big| \ll_m \frac{1}{1+x^2}   \ \ \ \ {\rm and} \ \ \ \  \big|f'_{2m+1}(x)\big| \ll_m \frac{1}{|x|(1+x^2)}.
\end{align*}
This follows from the fact that $f_{2m+1}$ and $f'_{2m+1}$ are bounded functions with power series representations 
\begin{align*}
f_{2m+1}(x)=\frac{1}{2m+1}\sum_{k=1}^{\infty}\dfrac{(-1)^{k-1}}{(2k+2m+1)x^{2k}}  \ \ \ {\rm and} \ \ \ f'_{2m+1}(x)=\dfrac{1}{2m+1}\displaystyle\sum_{k=1}^{\infty}\dfrac{(-1)^{k}(2k)}{(2k+2m+1)x^{2k+1}}
\end{align*}
for $|x| >1$. It then follows from \eqref{Eq_31_F} that
\begin{align}\label{bounds for big F}
\big|F_{\Delta}(x)\big| \ll_m \frac{\Delta^2}{\Delta^2+x^2}   \ \ \ \ {\rm and} \ \ \ \  \big|F'_{\Delta}(x)\big| \ll_m \frac{\Delta^2}{|x|(\Delta^2+x^2)}.
\end{align}
Observe that for any complex number $z$ we have 
\begin{align}\label{Fejer_bound}
\left|\frac{\sin \pi z}{ \pi z}\right|^2 \ll \frac{e^{2\pi |\im z|}}{1 + |z|^2}.\end{align}
Expressions \eqref{Def_G-} and \eqref{Def_G+} can be rewritten as
\begin{equation}\label{G-_rewritten}
G^{-}_{\Delta}(z)= \displaystyle\sum_{n=-\infty}^{\infty} \bigg(\dfrac{\sin\pi (z - n + \tfrac12)}{\pi (z - n + \tfrac12)}\bigg)^2  \left\{F_{\Delta}\big(n-\tfrac{1}{2}\big) + (z-n+\tfrac{1}{2}\big)F^{'}_{\Delta}\big(n-\tfrac{1}{2}\big)\right\}
\end{equation}
and 
\begin{equation}\label{G+_rewritten}
G^{+}_{\Delta}(z)= \bigg(\dfrac{\sin\pi z }{\pi z }\bigg)^2  F_{\Delta}(0) + \displaystyle\sum_{n\neq 0}\bigg(\dfrac{\sin\pi (z - n)}{\pi (z - n)}\bigg)^2  \left\{F_{\Delta}(n) + (z-n)F^{'}_{\Delta}(n)\right\}.
\end{equation}
It follows from \eqref{bounds for big F}, \eqref{Fejer_bound}, \eqref{G-_rewritten} and \eqref{G+_rewritten} that
\begin{align*}
\left|G^{\pm}_{\Delta}(z)\right| \ll_m \frac{\Delta^2}{1 + |z|}\,e^{2 \pi |\im z|}
\end{align*}
and from \eqref{Eq_50} this implies \eqref{EF_Lem_eq2}.

\medskip

To bound $G^{\pm}_{\Delta}$ on the real line, we explore the fact that $F_{\Delta}$ is an even function (and hence $F_{\Delta}'$ is odd) to group the terms conveniently. For the majorant we group the terms $n$ and $-n$ in \eqref{G+_rewritten} to get
\begin{align}\label{G^+_bound_real}
G^{+}_{\Delta}(x) & =\bigg(\dfrac{\sin\pi x}{\pi x}\bigg)^2 F_{\Delta}(0) +  \displaystyle\sum_{n=1}^{\infty} \bigg(\dfrac{\sin^2\pi (x - n)}{\pi^2 (x^2 - n^2)^2}\bigg) \Big\{(2x^2 + 2n^2)F_{\Delta}(n)  + (x^2 - n^2) \,2 n\, F^{'}_{\Delta}(n)\Big\}\,,
\end{align}
and it follows from \eqref{bounds for big F} and \eqref{Fejer_bound} that 
\begin{align}\label{Pf_Lem5_eq_1_g+}
\big|G^{+}_{\Delta}(x)\big| \ll_m \frac{\Delta^2}{\Delta^2 + x^2}. 
\end{align}
It may be useful to split the sum in \eqref{G^+_bound_real} into the ranges $\{n \leq |x|/2\}$, $\{|x|/2 < n \leq 2|x|\}$ and $\{2|x| < n\}$ to verify this last claim. The bound
\begin{align}\label{Pf_Lem5_eq_1_g-}
\big|G^{-}_{\Delta}(x)\big| \ll_m \frac{\Delta^2}{\Delta^2 + x^2}. 
\end{align} 
follows in an analogous way, grouping the terms $n$ and $1-n$ (for $n \geq 1$) in \eqref{G-_rewritten}. From \eqref{Eq_50}, \eqref{Pf_Lem5_eq_1_g+} and \eqref{Pf_Lem5_eq_1_g-} we arrive at \eqref{EF_Lem_eq1}.

\subsubsection{Part (ii)} From the inequalities \eqref{EF_Lem_eq1} and \eqref{EF_Lem_eq2}, it follows that the functions $g^{\pm}_{\Delta}$ have exponential type $2\pi\Delta$ and are integrable on $\mathbb{R}$. By the Paley-Wiener theorem, the Fourier transforms $\widehat{g}^{\pm}_{\Delta}$ are compactly supported on the interval $[-\Delta,\Delta]$. Moreover, using \eqref{EF_Lem_eq1} we obtain
\[
\big|\widehat{g}_{\Delta}^{\pm}(\xi)\big|=\Bigg|\int_{-\infty}^{\infty}g_{\Delta}^{\pm}(x)e^{-2\pi ix\xi}\,\dx\Bigg|\leq \int_{-\infty}^{\infty}\big|g_{\Delta}^{\pm}(x)\big|\,\dx \leq K_{2m+1} \int_{-\infty}^{\infty}\dfrac{1}{1+x^2}\,\dx\ll_m 1.
\]

\subsubsection{Part (iii)} From \eqref{Eq_31_F}, \eqref{Eq_50} and the identities in \cite[Section 11, Corollary 17 and Example 3]{CLV} we obtain
\begin{align*}
\int_{-\infty}^{\infty}& \big\{f_{2m+1}(x)-g^{-}_{2m+1,\Delta}(x)\big\}\,\dx \\
& =\dfrac{1}{\Delta}\int_{-\infty}^{\infty}\big\{F_{\Delta}(x)-G^{-}_{\Delta}(x)\big\}\,\dx \\
& = \dfrac{1}{\Delta}\int_{0}^{\infty}\Bigg\{\displaystyle\sum_{n \neq 0}(-1)^{n+1}\lambda^{-1/2}e^{-\pi\lambda^{-1}n^2}\Bigg\}\,\d\nu_{\Delta}(\lambda) \\
 & = \dfrac{1}{\Delta}\int_{0}^{\infty}\int_{1/2}^{3/2}\Bigg\{\displaystyle\sum_{n \neq 0}(-1)^{n+1}\lambda^{-1/2}e^{-\pi\lambda^{-1}n^2}\Bigg\}\left(\sigma-\tfrac{1}{2}\right)^{2m}\Bigg(\dfrac{e^{-\pi\lambda(\sigma-1/2)^{2} \Delta^2}-e^{-\pi\lambda\Delta^2}}{2\lambda}\Bigg)\,\d\sigma\, \d\lambda\\
 &  = \dfrac{1}{\Delta}\int_{1/2}^{3/2}\left(\sigma-\tfrac{1}{2}\right)^{2m}\int_{0}^{\infty}\Bigg\{\displaystyle\sum_{n \neq 0}(-1)^{n+1}\lambda^{-1/2}e^{-\pi\lambda^{-1}n^2}\Bigg\}\Bigg(\dfrac{e^{-\pi\lambda(\sigma-1/2)^{2} \Delta^2}-e^{-\pi\lambda\Delta^2}}{2\lambda}\Bigg)\,\d\lambda\, \d\sigma\\
 &  = \dfrac{1}{\Delta}\int_{1/2}^{3/2}\left(\sigma-\tfrac{1}{2}\right)^{2m}\log\Bigg(\dfrac{1+e^{-2\pi(\sigma-1/2)\Delta}}{1+e^{-2\pi\Delta}}\Bigg)\,\d\sigma,
\end{align*}
where the interchange of integrals is justified since the integrand is nonnegative. In a similar way, we have
\begin{align*}
 \int_{-\infty}^{\infty} & \big\{g^{+}_{2m+1,\Delta}(x)-f_{2m+1}(x)\big\}\,\dx \\
& =\dfrac{1}{\Delta}\int_{-\infty}^{\infty}\big\{G^{+}_{\Delta}(x)-F_{\Delta}(x)\big\}\,\dx \\
& = \dfrac{1}{\Delta}\int_{0}^{\infty}\Bigg\{\displaystyle\sum_{n \neq 0}\lambda^{-1/2}e^{-\pi\lambda^{-1}n^2}\Bigg\}\,\d\nu_{\Delta}(\lambda) \\
&  = \dfrac{1}{\Delta}\int_{0}^{\infty}\int_{1/2}^{3/2}\Bigg\{\displaystyle\sum_{n \neq 0}\lambda^{-1/2}e^{-\pi\lambda^{-1}n^2}\Bigg\}\left(\sigma-\tfrac{1}{2}\right)^{2m}\Bigg(\dfrac{e^{-\pi\lambda(\sigma-1/2)^{2} \Delta^2}-e^{-\pi\lambda\Delta^2}}{2\lambda}\Bigg)\,\d\sigma\, \d\lambda\\
&  = \dfrac{1}{\Delta}\int_{1/2}^{3/2}\left(\sigma-\tfrac{1}{2}\right)^{2m}\int_{0}^{\infty}\Bigg\{\displaystyle\sum_{n \neq 0}\lambda^{-1/2}e^{-\pi\lambda^{-1}n^2}\Bigg\}\Bigg(\dfrac{e^{-\pi\lambda(\sigma-1/2)^{2} \Delta^2}-e^{-\pi\lambda\Delta^2}}{2\lambda}\Bigg)\,\d\lambda \,\d\sigma\\
& = -\dfrac{1}{\Delta}\int_{1/2}^{3/2}\left(\sigma-\tfrac{1}{2}\right)^{2m}\log\Bigg(\dfrac{1-e^{-2\pi(\sigma-1/2)\Delta}}{1-e^{-2\pi\Delta}}\Bigg)\d\sigma.
\end{align*}
This concludes  the proof of Lemma \ref{lema extremal}.

\section{Extension to general $L$-functions}\label{Sec_L_functions}

\subsection{Statement} In this section we briefly discuss how to extend our results to a general family of $L$-functions in the framework of \cite[Chapter 5]{IK}. Below we adopt the notation
$$\Gamma_\mathbb R(z):=\pi^{-z/2}\,\Gamma\left(\frac{z}{2}\right),$$ 
where $\Gamma$ is the usual Gamma function. We consider a meromorphic function $L(\cdot,\pi)$ on $\mathbb C$ which meets the following requirements (for some positive integer $d$ and some $\vartheta\in[0,1]$). The examples include the Dirichlet $L$-functions $L(\cdot,\chi)$ for primitive characters $\chi$.

\medskip

\noindent(i) There exists a sequence $\{\lambda_\pi(n)\}_{n\ge1}$ of complex numbers ($\lambda_\pi(1) =1$) such that the series $$\sum_{n=1}^\infty\frac{\lambda_\pi(n)}{n^s}$$ converges absolutely to $L(s,\pi)$ on $\{s\in\mathbb C \,;\,\text{Re}\,s>1\}$.

\medskip

\noindent(ii) For each prime number $p$, there exist $\alpha_{1,\pi}(p),\alpha_{2,\pi}(p),\ldots,\alpha_{d,\pi}(p)$ in $\mathbb C$ such that $|\alpha_{j,\pi}(p)|\leq p^\vartheta$, where $0 \leq \vartheta \leq 1$ is independent of $p$, and 
$$L(s,\pi)=\prod_p\prod_{j=1}^d\left(1-\frac{\alpha_{j,\pi}(p)}{p^s}\right)^{-1},$$
with absolute convergence on the half plane $\{s \in \C; \text{Re}\,s>1\}$.

\medskip

\noindent(iii) For some positive integer $N$ and some complex numbers $\mu_1,\mu_2,\ldots,\mu_d$ whose real parts are greater than $-1$ and such that $\{\mu_1,\mu_2,\ldots,\mu_d\}=\{\overline{\mu_1},\overline{\mu_2},\ldots,\overline{\mu_d}\}$, the completed $L$-function 
$$\Lambda(s,\pi):=N^{s/2}\prod_{j=1}^d \Gamma_\mathbb R(s+\mu_j)L(s,\pi)$$
is a meromorphic function of order 1 that has no poles other than $0$ and $1$. The points $0$ and $1$ are poles with the same order $r(\pi)\in\{0,1,\ldots,d\}$. Furthermore, the function $\Lambda(s, \tilde\pi):=\overline{\Lambda(\overline s, \pi)}$ satisfies the functional equation
\begin{equation}\label{Intro_FE}
\Lambda(s,\pi)=\kappa\,\Lambda(1-s,\tilde\pi)
\end{equation}
for some unitary complex number $\kappa$.

\medskip

We define the {\it analytic conductor} of $L(\cdot,\pi)$ as the function 
$$C(t,\pi)=N\prod_{j=1}^d(|it+\mu_j|+3).$$
In what follows we assume the {\it generalized Riemann hypothesis}, which asserts that $\Lambda(s,\pi)\neq 0$ if ${\rm Re}\,s\neq\tfrac12$. For $t>0$, we define here the moments of the argument function as the sequence, for $n\geq 0$,
\begin{equation*}
S_n(t,\pi) = -\frac{1}{\pi} \,\,\im{\left\{\dfrac{i^{n}}{n!}\int_{1/2}^{\infty}{\left(\sigma-\tfrac{1}{2}\right)^{n}\,\frac{L'}{L}(\sigma+it, \pi)}\,\d\sigma\right\}}.
\end{equation*}
Differentiating under the integral sign and using integration by parts, one can see that $S_n'(t,\pi) = S_{n-1}(t,\pi)$ for $t>0$ (in the case $n=1$ we may restrict ourselves to the case when $t$ is not the ordinate of a zero of $L$). 

\medskip

The main result of this section is the following.

\begin{theorem} \label{Thm6}
For $n\geq 0$, let $C_n^{\pm}$ be the constants defined in Theorem \ref{Thm1}. Let $L(\cdot,\pi)$ satisfy the generalized Riemann hypothesis. Then, for all $t>0$ we have
\begin{align*}
-\Big((1 + 2\vartheta)^{n+1} C_n^- + o(1)\Big) \frac{\log C(t,\pi)}{(\logfeio)^{n+1}} \leq S_n(t,\pi)  \leq \Big((1 + 2\vartheta)^{n+1} C_n^+ + o(1)\Big) \frac{\log C(t,\pi)}{(\logfeio)^{n+1}}.
\end{align*}
The terms $o(1)$ above are $O(\log \logfeio / \logfeio)$, where the constant implicit by the $O$-notation may depend on $n$ but does not depend on $d$ or $N$.  
\end{theorem}

The case $n=0$ of this theorem was established in \cite{CCM2} and the case $n=1$ was established in \cite{CF}.

\subsection{Outline of the proof} The proof of Theorem \ref{Thm6} follows the same circle of ideas used to prove Theorem \ref{Thm1}. We only give here a brief account of the proof, indicating the changes that need to be made.  Notice that we only need to prove Theorem \ref{Thm6} for the case $n$ odd, since the case of $n\geq 2$ even follows by reproducing the interpolation argument of Section \ref{Int_sec}.

\medskip

By the product expansion of $L(\cdot,\pi)$ and the inequality $|\alpha_{j,\pi}(p)|\leq p$, 
\begin{equation}\label{L_fun_eq1}
|\log|L(s,\pi)||\leq d\log\zeta({\rm Re}\,s-1)=O\bigg(\frac{d}{2^{{\rm Re}\,s}}\bigg)
\end{equation} 
for any $s$ with ${\rm Re}\,s\geq\frac{5}{2}$. Note also that 
\begin{equation*}
\frac{L'}{L}(s,\pi)=-\sum_p\sum_{j=1}^d\frac{\alpha_{j,\pi}(p)}{p^s}\left(1-\frac{\alpha_{j,\pi}(p)}{p^s}\right)^{-1}\log p\,,
\end{equation*}
where the right-hand side converges absolutely if ${\rm Re}\,s>1$. This shows that the logarithmic derivative of $L(\cdot,\pi)$ has a Dirichlet series
\begin{equation}\label{L_fun_eq2}
\frac{L'}{L}(s,\pi)=-\sum_{n=2}^\infty\frac{\Lambda_\pi(n)}{n^s},
\end{equation}
where $\Lambda_\pi(n)=0$ if $n$ is not a power of prime and $\Lambda_\pi(p^k)=\sum_{j=1}^d\alpha_{j,\pi}(p)^k\log p$ if $p$ is prime and $k$ is a positive integer. If follows that 
\begin{equation}\label{L_fun_eq3}
\big|\Lambda_\pi(n)\big| \leq d \,\Lambda(n) \,n^{\vartheta}.\end{equation}
Let $f_n$ be defined by \eqref{Def_f_2m} - \eqref{Def_f_2m+1} and consider here the dilated functions $$\widetilde{f}_n(x) = 2^n f_n\left(\tfrac{x}{2}\right).$$
The following result is the analogue of Lemma \ref{Rep_lem}.

\begin{lemma}
Let $L(s,\pi)$ satisfy the generalized Riemann hypothesis. For each $n\geq 0$ and $t>0$ $($and $t$ not coinciding with an ordinate of a zero of $L(s,\pi)$ in the case $n=0$$)$ we have:
\begin{itemize}
\item[(i)] If $n = 2m$, for $m \in \Z^+$, then
\begin{equation} \label{L_Lem2_eq_representation_even}
S_{2m}(t,\pi)= \frac{(-1)^{m}}{\pi(2m)!} \,\sum_{\gamma}\widetilde{f}_{2m}(t-\gamma)\,+\,O(d).  
\end{equation}
\item[(ii)] If $n = 2m+1$, for $m \in \Z^+$, then
\begin{equation} \label{L_Lem2_eq_representation_odd}
S_{2m+1}(t,\pi)=\frac{(-1)^{m}\ 2^{2m+1}}{\pi(2m+2)!}\log C(t,\pi)  -\frac{(-1)^{m}}{\pi(2m)!} \sum_{\gamma}\widetilde{f}_{2m+1}(t-\gamma) \,+ \,O(d).
\end{equation}
\end{itemize}
The sums in \eqref{L_Lem2_eq_representation_even} and \eqref{L_Lem2_eq_representation_odd} run over all values $\gamma$ such that $\Lambda(\tfrac12 + i \gamma, \pi) = 0$, counted with multiplicity.
\end{lemma}
\begin{proof} 
This follows the outline of the proof of Lemma \ref{lema extremal}, truncating the integrals \eqref{expansion_S_m_odd} and \eqref{Lem_2_case_2_S_2m} in the point $5/2$ instead of $3/2$, and introducing the test point $5/2 + it$ instead of $3/2 + it$ in \eqref{test_point_1} and \eqref{test_point_2}. This is due to \eqref{L_fun_eq1}, \eqref{L_fun_eq2} and \eqref{L_fun_eq3}, in order to better deal with the absolute convergence issues, and ultimately causes the replacement of $f_n$ by the dilated version $\widetilde{f}_n$. The Hadamard factorization \eqref{Had_fact} and the partial fraction decomposition \eqref{cota} should be replaced by their $L$-function analogues \cite[Theorem 5.6 and Proposition 5.7]{IK} and \cite[Equation 5.24]{IK}, respectively. Full details are given in \cite[Section 4.2]{CCM2} for $n=0$ and in \cite[Lemma 4]{CF} for $n=1$.
\end{proof}
The explicit formula for $L$-functions takes the following form.
\begin{lemma}[Explicit formula for $L$-functions]\label{Exp_for_L}
Let $h(s)$ be analytic in the strip $|\im{s}|\leq \tfrac12+\varepsilon$ for some $\varepsilon>0$, and assume that $|h(s)|\ll(1+|s|)^{-(1+\delta)}$ for some $\delta>0$ when $|\re{s}|\to\infty$. Then 
\begin{align*}
\begin{split}
\sum_{\rho} h\left(\frac{\rho - \tfrac12}{i}\right)&= r(\pi)\left\{h\left(\frac{1}{2i}\right)+h\left(-\frac{1}{2i}\right)\right\} + \frac{\log N}{2\pi}\int_{-\infty}^\infty h(u)\dd u\\
& +\frac{1}{\pi}\sum_{j=1}^d\int_{-\infty}^\infty h(u)\,{\rm Re}\,\frac{\Gamma_\mathbb R'}{\Gamma_\mathbb R}\left(\hh+\mu_j+iu\right)\d u\\
& -\frac{1}{2\pi}\sum_{n=2}^\infty\frac{1}{\sqrt{n}}\left\{\Lambda_{\pi}(n)\, \widehat h\left(\frac{\log n}{2\pi}\right)+\overline{\Lambda_{\pi}(n)}\, \widehat h\left(\frac{-\log n}{2\pi}\right)\right\} \\
&- \!\!\sum_{-1 < {\rm Re}\,\mu_j<-\meio}\!\!\left\{h\left(\frac{-\mu_j-\meio}{i}\right)\!+h\left(\frac{\mu_j+\meio}{i}\right)\!\right\} -\meio\sum_{{\rm Re}\,\mu_j=-\meio} \left\{h\left(\frac{-\mu_j-\meio}{i}\right)\!+h\left(\frac{\mu_j+\meio}{i}\right)\!\right\},
\end{split}
\end{align*}
where the sum runs over all zeros $\rho$ of $\Lambda(\cdot, \pi)$ and the coefficients $\Lambda_\pi(n)$ are defined by \eqref{L_fun_eq2}.
\end{lemma}
\begin{proof}
This is a modification of the proof of \cite[Theorem 5.12]{IK}. The idea is to consider the integral
$$\frac{1}{2\pi i } \oint h\!\left(\frac{s- \tfrac12}{i}\right) \frac{\Lambda'(s,\pi)}{\Lambda(s, \pi)}\,\ds$$
over the rectangular contour connecting the points $1 + \eta + iT_1, - \eta + iT_1, - \eta - iT_2, 1 + \eta - iT_2$, say with $\eta = \varepsilon/2$. Then one sends $T_1, T_2 \to \infty$ over an appropriate sequence of heights that keep the zeros as far as possible (recall that at height $T$, we have $O(\log C(t,\pi))$ zeros, see \cite[Proposition 5.7]{IK}). One then uses the functional equation \eqref{Intro_FE} to replace the integral over the line $\re s = -\eta$ by an integral over the line $\re s = 1 + \eta$, and finally one moves the remaining integrals to the line $\re s = \hh$, picking up possibly some additional poles at the $\mu_j$'s.
\end{proof}
\subsubsection{Conclusion of the proof} For $n = 2m+1$ we have the extremal majorants and minorants of exponential type $\Delta$ for $\widetilde{f}_{2m+1}$ given by Lemma \ref{lema extremal} . These are
$$\widetilde{g}^+_{2m+1, \Delta}(z) : = 2^{2m+1} g^{+}_{2m+1, 2\Delta}(z/2) \ \ \ {\rm and} \ \ \ \widetilde{g}^-_{2m+1, \Delta}(z) : = 2^{2m+1} g^{-}_{2m+1, 2\Delta}(z/2).$$
We now replace $\widetilde{f}_{2m+1}$ in \eqref{L_Lem2_eq_representation_odd} and evaluate using the explicit formula. Let us consider, for instance, the upper bound in the case where $m$ is odd. Letting $h(z):= \widetilde{g}^+_{2m+1, \Delta}(t-z)$ we have
\begin{align}\label{Final_eq1}
S_{2m+1}(t,\pi)\leq -\frac{ 2^{2m+1}}{\pi(2m+2)!}\log C(t,\pi)  +\frac{1}{\pi(2m)!} \sum_{\gamma}h(\gamma) \,+ \,O(d).
\end{align}
We evaluate $\sum_{\gamma}h(\gamma)$ from the explicit formula (Lemma \ref{Exp_for_L}). From Lemma \ref{lema extremal} we have
\begin{align}\label{Final_eq2}
\begin{split}
&\Big|r(\pi)\left\{h\left(\tfrac{1}{2i}\right) +h\left(-\tfrac{1}{2i}\right)\right\}\Big| \\
&\ \ \ \ \ \ \ \ \  + \Bigg|\sum_{-1 < {\rm Re}\,\mu_j<-\meio}\!\!\left\{h\left(\tfrac{-\mu_j-\meio}{i}\right)\!+h\left(\tfrac{\mu_j+\meio}{i}\right)\!\right\} +\meio\sum_{{\rm Re}\,\mu_j=-\meio} \left\{h\left(\tfrac{-\mu_j-\meio}{i}\right)\!+h\left(\tfrac{\mu_j+\meio}{i}\right)\!\right\}\Bigg| \\
 & \ll_m d\,\Delta^2\,e^{\pi\Delta}.
 \end{split}
 \end{align}
Using Striling's formula in the form
\begin{equation*}
\frac{\Gamma_\mathbb R'}{\Gamma_\mathbb R}(z)=\frac{1}{2}\log(2+z)-\frac{1}{z}+O(1),
\end{equation*} 
valid for ${\rm Re}\,z>-\meio$, we find that
\begin{align}\label{Final_eq3}
\frac{\log N}{2\pi}\int_{-\infty}^\infty h(u)\dd u & +\frac{1}{\pi}\sum_{j=1}^d\int_{-\infty}^\infty h(u)\,{\rm Re}\,\frac{\Gamma_\mathbb R'}{\Gamma_\mathbb R}\left(\hh+\mu_j+iu\right)\d u = \frac{\log C(t,\pi)}{2\pi}\int_{-\infty}^\infty h(u)\dd u + O(d).
\end{align}
By Lemma \ref{lema extremal}, the Fourier transform $\widehat h(\xi)$ is supported on $[-\Delta,\Delta]$ and is uniformly bounded. Also, $|\Lambda_\pi(n)|\le d\,\Lambda(n)\,n^{\vartheta}$, and therefore
\begin{align}\label{Final_eq4}
\frac{1}{2\pi}\sum_{n=2}^\infty\frac{1}{\sqrt{n}}\!\left\{\Lambda_{\pi}(n)\,\widehat h\left(\frac{\log n}{2\pi}\right)\!+\!\overline{\Lambda_{\pi}(n)}\,\widehat h\left(\frac{-\log n}{2\pi}\right)\right\} \!= O\left(\!d\sum_{n\leq e^{2\pi\Delta}}\Lambda(n)n^{\vartheta-\meio}\!\right) = O\left(d\,e^{(1+2\vartheta)\pi\Delta}\right),
\end{align}
where the last equality follows by the Prime Number Theorem and summation by parts.

\medskip

From the computations in \eqref{Final_eq0} and \eqref{ASA_eq4}, together with \eqref{Final_eq1}, \eqref{Final_eq2}, \eqref{Final_eq3} and \eqref{Final_eq4} we get
\begin{align*}
S_{2m+1}(t,\pi)\leq \frac{C_{2m+1}^+}{(\pi \Delta)^{2m+2}}  \log C(t,\pi) + O\Big(e^{-2\pi \Delta} \log C(t,\pi) \Big) +  O\left(d\,\Delta^2\, e^{(1+2\vartheta)\pi\Delta}\right).
\end{align*}
for any $t>0$ and any $\Delta \geq 1$. Choosing 
$$\pi \Delta = \max\left\{\frac{\logfeio - (2m+5)\log \logfeio}{(1+2\vartheta)} \ , \ \pi\right\}$$
yields the desired result. 

\medskip

The lower bound for $m$ odd is analogous, using the minorant $\widetilde{g}^-_{2m+1, \Delta}$. The upper and lower bounds for $m$ even are also analogous, changing the roles of $\widetilde{g}^+_{2m+1, \Delta}$ and $\widetilde{g}^-_{2m+1, \Delta}$. 

\medskip

We refer the interested reader to \cite{CF}, where full details are given for the case $n=1$. 

\section*{Acknowledgements}
E.C. acknowledges support from CNPq-Brazil grants $305612/2014-0$ and $477218/2013-0$, and FAPERJ grant $E-26/103.010/2012$. A.C. acknowledges support from CNPq-Brazil. We would like to thank Micah Milinovich and Vorrapan Chandee for the insightful conversations on the topic.

\end{document}